\documentclass[a4paper,draft]{amsproc}
\usepackage{amssymb}
\usepackage[hyphens]{url} 

\theoremstyle{plain}
 \newtheorem{theorem}{Theorem}[section]
 \newtheorem{proposition}{Proposition}[section]
 \newtheorem{lemma}{Lemma}[section]
 \newtheorem{corollary}{Corollary}[section]
\theoremstyle{definition}

\theoremstyle{remark}

 \numberwithin{equation}{section}

\renewcommand{\leq}{\leqslant}
\renewcommand{\geq}{\geqslant}

\title[ A cotangent sum related to the Estermann zeta
function/ M. Goubi]{Series representation of a cotangent sum related
to the Estermann zeta function}



\author{\bfseries Mouloud  Goubi} 
\address{Mouloud Goubi\\
Department of Mathematics \\
University of UMMTO RP. 15000\\
Tizi-ouzou, Algeria\\
Laboratoire d'Alg\`ebre et Th\'eorie des Nombres, USTHB Alger}
\email{mouloud.goubi@ummto.dz}

\begin{document}

\vspace{18mm} \setcounter{page}{1} \thispagestyle{empty}

\begin{abstract}
In this paper, we are interested by the cotangent sum
$c_0\left(\frac{q}{p}\right)$ related to the Estermann zeta function
for the special case when $q=1$ and get explicit formula for its
series expansion, which represents an improvement of the identity
$(2. 1)$ Theorem $(2.1)$ in the recent work \cite{GOUBI4}.
\end{abstract}

\maketitle

\section{Introduction and main results}
For $p$ a positive integer and $q=1,2,\cdots,p-1$ such that
$\left(p,q\right)=1$, let the cotangent sum \cite{Rassias}
\begin{equation}\label{equa1}
c_0\left(\frac{q}{p}\right)=-\sum_{k=1}^{p-1}\frac{k}{p}\cot\frac{\pi
kq}{p}.
\end{equation}
$c_0\left(\frac{q}{p}\right)$ is the value at
$s=0$,\[E_0\left(0,\frac{q}{p}\right)=\frac{1}{4}+\frac{i}{2}c_0\left(\frac{q}{p}\right)\]
of the Estermann zeta function
\[E_0\left(s,\frac{q}{p}\right)=\sum_{k\geq1}\frac{d(k)}{k^s}\exp\left(\frac{2\pi
ikq}{p}\right).\]\\
This sum is related directly to Vasyunin cotangent sum \cite{VASY};
\begin{equation}\label{equa2}
V\left(\frac{q}{p}\right)=\sum_{r=1}^{p-1}\left\{\frac{rq}{p}\right\}\cot\left(\frac{\pi
r}{p}\right)=-c_0\left(\frac{\overline{q}}{p}\right)
\end{equation}
which arises in the study of the Riemann zeta function by virtue of
the formula \cite{Bettin,Maier}.
\begin{eqnarray}\label{zetaint}
\begin{split}
&\qquad \frac{1}{2\pi\sqrt{pq}}\int_{-\infty}^{+\infty}\left|\zeta\left(\frac{1}{2}+it\right)\right|^2\left(\frac{q}{p}\right)^{it}\frac{dt}{\frac{1}{4}+t^2}=\\
&\frac{\log2\pi-\gamma}{2}\left(\frac{1}{p}+\frac{1}{q}\right)+\frac{p-q}{2pq}\log\frac{q}{p}-\frac{\pi}{2pq}
\left(V\left(\frac{p}{q}\right)+V\left(\frac{q}{p}\right)\right).
\end{split}
\end{eqnarray}
This formula is connected to the approach of Nyman, Beurling and
B\'{a}ez-Duarte to the Riemann hypothesis \cite{Maier1,Covalanko}.
Which states that the Riemann hypothesis is true if and only if
$\displaystyle\lim_{n\to\infty}d_N=0$, where
\[d^2_N=\inf_{A_N}\frac{1}{2\pi}\int_{-\infty}^{+\infty}\left|1-\zeta A\left(\frac{1}{2}+it\right)\right|^2\frac{dt}{\frac{1}{4}+t^2}\]
and the infimum is taken over all Dirichlet polynomials
\[A_N(s)=\sum_{n=1}^{N}\frac{a_n}{n^s}.\]

In the literature; different results about $V(\frac{q}{p})$ and
$c_0(\frac{q}{p})$ are obtained. For more details we refer to
\cite{GOUBI3, Rassias, Maier, Bettin2, GOUBI1} and reference
therein.\\

Exactly our interest in this work is the case $q=1$ in order to
compute explicitly the sequence $b_k$ in the series expansion $(2.
1)$ Theorem $(2.1)$ \cite{GOUBI4} of $c_0(\frac{1}{p})$:
\begin{equation}\label{equac0}
c_0\left(\frac{1}{p}\right)=\frac{p\left(p-1\right)\left(p-2\right)}{\pi}\sum_{k\geq0}\frac{b_k}{\left(k+1\right)
\left(k+p+1\right)\left(k+2\right)\left(k+p\right)}
\end{equation}
where $b_k$ is generated by the function
$f\left(x\right)=\frac{1}{\left(1-x\right)^2\left(1-x^p\right)}=\sum_{k\geq0}b_kx^k$.
The first terms are $b_0=1, b_1=2$ and the others are given by the
recursive formulae:
\begin{equation}\label{recurse1}
b_k-2b_{k-1}+b_{k-2}=0,\ 2\leq k\leq p-1,\ k=p+1,
\end{equation}
\begin{equation}\label{recurse2}
b_p-2b_{p-1}+b_{p-2}=1
\end{equation}
and
\begin{equation}\label{recurse3}
b_k-2b_{k-1}+b_{k-2}-b_{k-p}+2b_{k-p-1}-b_{k-p-2}=0,\ k\geq p+2.
\end{equation}

\section{Statement of main results}
In the following theorem, we prove that
$b_k=\left(k+1-\frac{p}{2}\left\lfloor\frac{k}{p}\right
\rfloor\right)\left(\left\lfloor\frac{k}{p}\right\rfloor+1\right)$,
where $\left\lfloor.\right\rfloor$ is the well known floor function.
\begin{theorem}\label{th1}
\begin{equation}\label{equath1}
c_0\left(\frac{1}{p}\right)=\frac{p\left(p-1\right)\left(p-2\right)}{\pi}\sum_{k\geq0}\frac{\left(k+1-\frac{p}{2}\left\lfloor\frac{k}{p}\right
\rfloor\right)\left(\left\lfloor\frac{k}{p}\right\rfloor+1\right)}{\left(k+1\right)\left(k+p+1\right)\left(k+2\right)\left(k+p\right)}.
\end{equation}
\end{theorem}
Let the integer sequence $s_n(p)$ defined by
\[s_n(p)=p\left(p-1\right)\left(p-2\right)\sum_{k=0}^{n}\frac{\left(k+1-\frac{p}{2}\left
\lfloor\frac{k}{p}\right\rfloor\right)\left(\left\lfloor\frac{k}{p}\right\rfloor+1\right)}{\left(k+1\right)\left(k+p+1\right)\left(k+2\right)\left(k+p\right)}.\]
Then  $\pi c_0\left(\frac{1}{p}\right)=\lim_{n\to\infty}s_n(p)$
which explains that $\pi c_0\left(\frac{1}{p}\right)$ lies to
$\overline{\mathbb{Q}}.$ Only in means of the Theorem \ref{th1}, for
$q=1$ the identity \eqref{zetaint} becomes
\begin{eqnarray}
\begin{split}
&\qquad \frac{1}{2\pi\sqrt{p}}\int_{-\infty}^{+\infty}\left|\zeta\left(\frac{1}{2}+it\right)\right|^2\left(\frac{1}{p}\right)^{it}\frac{dt}{\frac{1}{4}+t^2}=\\
&\frac{p-1}{2p}\left[2\left(\log2\pi-\gamma\right)-\log
p+p\left(p-2\right)\sum_{k\geq0}\frac{\left(k+1-\frac{p}{2}\left\lfloor\frac{k}{p}\right
\rfloor\right)\left(\left\lfloor\frac{k}{p}\right\rfloor+1\right)}{\left(k+1\right)\left(k+p+1\right)\left(k+2\right)\left(k+p\right)}\right].
\end{split}
\end{eqnarray}
Let the arithmetical function of two variables
\[\theta(i,j)=\frac{pi^2+\left(p+2r+2\right)i+2r+2}
{\left(ip+r+1\right)\left(ip+p+r+1\right)\left(ip+r+2\right)\left(ip+p+r\right)}\]
then the following proposition is obtained
\begin{proposition}\label{propo1}
\begin{equation}\label{equapropo1}
c_0\left(\frac{1}{p}\right)=\frac{p\left(p-1\right)\left(p-2\right)}{2\pi}\sum_{i\geq0}\sum_{r=0}^{r-1}\theta(i,j)
\end{equation}
\end{proposition}

\section{Proof of main results}
\subsection{Proof of Theorem~\ref{th1}}
We take inspiration from the theory of generating functions
\cite{GOSPA,GOUBI2}, and prove that the sequence $b_k$ generated by
the rational function $f(x)$ is explicitly done in the following
lemma. And the identity \eqref{equath1}, Theorem \ref{th1} is
deduced.
\begin{lemma}\label{genera}
\begin{equation}\label{gene}
\frac{1}{\left(1-x\right)^2\left(1-x^p\right)}=\sum_{k\geq0}b_kx^k,\
|x|<1
\end{equation}
with
\begin{equation}\label{gene1}
b_k=\left(k+1-\frac{p}{2}\left\lfloor\frac{k}{p}\right\rfloor\right)\left(\left\lfloor\frac{k}{p}\right\rfloor+1\right)
\end{equation}
\end{lemma}
\begin{proof}
It is well known that
\begin{equation}\label{eqprof1}
\frac{1}{1-x}=\sum_{k\geq0}x^k,\ |x|<1
\end{equation}
and
\begin{equation}\label{eqprof2}
\frac{1}{1-x^p}=\sum_{k\geq0}b_p(k)x^k,\ |x|<1.
\end{equation}
with $b_p(k)=1$ if $p$ divides $k$ and zero otherwise. Then
\begin{equation*}
\frac{1}{\left(1-x\right)^2}=\sum_{k\geq0}\left(k+1\right)x^k,\
|x|<1
\end{equation*}
and
\begin{equation*}
\frac{1}{\left(1-x\right)^2\left(1-x^p\right)}=\sum_{k\geq0}\sum_{j=0}^{k}\left(k-j+1\right)b_p(j)x^k,\
|x|<1.
\end{equation*}
But \[\left\lfloor\frac{k}{p}\right\rfloor p\leq
k<\left(\left\lfloor\frac{k}{p}\right\rfloor+1\right) p\] then
\[\sum_{j=0}^{k}\left(k-j+1\right)b_p(j)=\sum_{j=0}^{\left\lfloor\frac{k}{p}\right\rfloor}\left(k-jp+1\right)\]
Since
\[\sum_{j=0}^{\left\lfloor\frac{k}{p}\right\rfloor}\left(k-jp+1\right)=\left(k+1\right)\left(\left\lfloor\frac{k}{p}\right\rfloor+1\right)
-\frac{p}{2}\left\lfloor\frac{k}{p}\right\rfloor\left(\left\lfloor\frac{k}{p}\right\rfloor+1\right).\]
Then
\[b_k=\left(k+1\right)\left(\left\lfloor\frac{k}{p}\right\rfloor+1\right)
-\frac{p}{2}\left\lfloor\frac{k}{p}\right\rfloor\left(\left\lfloor\frac{k}{p}\right\rfloor+1\right).\]
Finally
\[b_k=\left(k+1-\frac{p}{2}\left\lfloor\frac{k}{p}\right\rfloor\right)\left(\left\lfloor\frac{k}{p}\right\rfloor+1\right)\]
 \end{proof}

 \subsection{Proof of Proposition \ref{propo1}.}
 In this subsection, we expose two methods to prove the Proposition
 \ref{propo1}.\\

 {\bf Analytic method.} We began by the following interesting lemma
 \begin{lemma}\label{lembk}
 \begin{eqnarray}\label{bk}
b_k= \left\{
\begin{array}{lll}
k+1\ &\quad \textrm{if}\ k<p,\\
b_{k-p}+k+1\ &\quad  \textrm{if}\ k\geq p.
\end{array}
\right.
\end{eqnarray}
 \end{lemma}
 \begin{proof}
We remember for $|x|<1$ that
$f(x)=\frac{1}{\left(1-x\right)^2\left(1-x^p\right)}=\sum_{k\geq0}b_kx^k$.
It is well known that
$f^{(k)}(0)=\frac{d^kf(x)}{dx^k}|_{x=0}=k!b_k$. Taking
$g(x)=\left(1-x\right)^{-2}$, it is easy to show that
$g^{(k)}(x)=\frac{(k+1)!}{(1-x)^{k+2}}$.

Using Leibnitz formula for successive derivatives of any infinitely
derivable function explained in the work \cite{GOUBI2}, and the
identity $(1-x^p)f(x)=(1-x)^2$ we deduce that
\[\frac{d^k(1-x^p)f(x)}{dx^k}=\frac{(k+1)!}{(1-x)^{k+2}}.\]
But
 \[\frac{d^k(1-x^p)f(x)}{dx^k}=\sum_{j=0}^{k}{k\choose
 j}(1-x^p)^{(j)}f^{k-j}(x).\]
 Since $(1-x^p)^{(j)}=-(p)_j\sigma_p(j)x^{p-j}$, where
 $\sigma_p(j)=1$ if $j\leq p$ and zero otherwise, then
 \[(1-x^p)f^{(k)}(x)-\sum_{j=1}^{k}{k\choose
 j}(p)_j\sigma_p(j)x^{p-j}f^{k-j}(x)=\frac{(k+1)!}{(1-x)^{k+2}}.\]
Furthermore $f^{(k)}(0)=(k+1)!$ for $k<p$ and for $k\geq p$ we have
\[f^{(k)}(0)-{k\choose
 p}p!f^{(k-p)}(0)=(k+1)!.\] Which means that $b_{k}=k+1$ for $k<p$
 and $b_k=b_{k-p}+k+1$ for $k\geq p.$
\end{proof}
\begin{corollary}\label{corobk} For $0\leq r\leq p-1$ and $i\geq0$ we have
\begin{equation}\label{equacorobk}
b_{ip+r}=\frac{1}{2}\left(pi^2+\left(p+2r+2\right)i+2r+2\right)
\end{equation}
\end{corollary}
\begin{proof}
First in means of the formula \eqref{bk} we have $b_r=r+1$,
$b_{p+r}=p+2r+2$, $b_{2p+r}=3p+3r+3$, $b_{3p+r}=6p+4r+4$ and
$b_{4p+r}=10p+5r+5.$ Then for $i\geq0$ suppose that
$b_{ip+r}=\frac{1}{2}i(i+1)p+(i+1)(r+1)$ then
\[b_{(i+1)p+r}=b_{ip+r}+(i+1)p+r+1=\frac{1}{2}i(i+1)p+(i+1)(r+1)+(i+1)p+r+1.\]
Finally
\[b_{(i+1)p+r}=\frac{1}{2}(i+1)(i+2)p+(i+2)(r+1)\] and the result
follows.
\end{proof}
Combining the identities in Lemma \ref{lembk} and Corollary
\ref{corobk} we get the desired result \eqref{equapropo1} in
Proposition \ref{propo1}.\\

{\bf Arithmetic method.} This is not expensive, just writing
$k=ip+r$ which is the Euclidean division of $k$ over $p$ then
$\left\lfloor\frac{k}{p}\right\rfloor=i$. Substituting this value in
the expression of $b_k$ in Lemma \ref{genera}, we get
\[b_{ip+r}=\left(\frac{p}{2}i+r+1\right)\left(i+1\right).\] After
development we get the result \eqref{equacorobk} of Corollary
\ref{corobk}, which implies the result \eqref{equapropo1}
Proposition \ref{propo1}.

\section{Consequences}
 \subsection{Limit at infinity of $c_0(\frac{1}{p})$}
 \begin{theorem}\label{th2}
 \begin{equation}
 0\leq\lim_{p\to+\infty}\frac{1}{p^3}c_0\left(\frac{1}{p}\right)\leq\frac{1}{2}
 \end{equation}
 \end{theorem}
 An interesting open question is how about the exact value of this
 limit. Since $c_0\left(\frac{1}{p}\right)$ is related directly to Riemann hypothesis where the real part of the zeros of the
 Riemann zeta function is $\frac{1}{2}$, one can conjecture that
 $\lim_{p\to+\infty}\frac{1}{p^3}c_0\left(\frac{1}{p}\right)=\frac{1}{2}$.
 \subsection{Proof of Theorem \ref{th2}}
We also use notational convention $0^0=1$ and the functions
\begin{equation}
\varphi_m(r,p)=\sum_{i\geq0}\frac{i^m}
{\left(ip+r+1\right)\left(ip+p+r+1\right)\left(ip+r+2\right)\left(ip+p+r\right)}
\end{equation}

It's clear that $\varphi_m(r,p)$ converge for
$i\in\left\{0,1,2\right\}$ and diverge for the others. In means of
these functions another expression of $c_0\left(\frac{1}{p}\right)$
is deduced by using the identity \eqref{equacorobk} of the last
Corollary \ref{corobk}:
\begin{equation*}
c_0\left(\frac{1}{p}\right)=\frac{p\left(p-1\right)\left(p-2\right)}{2\pi}\sum_{r=0}^{p-1}\left[p\varphi_2(r,p)+(p+2r+2)\varphi_1(r,p)+
(2r+2)\varphi_0(r,p)\right]
\end{equation*}
In the following lemma, we develop some inequalities satisfied by
the functions $\varphi_i(r,p)$ for $i\in\left\{0,1,2\right\}.$
\begin{lemma}\label{lemfurth1}
\begin{equation}
\frac{1}{2p^2(p+1)(2p-1)}+\frac{1}{81p^4}\zeta(4)<\varphi_0(r,p)<\frac{1}{2p(p+1)}+\frac{1}{p^4}\zeta(4)
\end{equation}
\begin{equation}
\frac{1}{81p^4}\zeta(3)<\varphi_1(r,p)<\frac{1}{p^4}\zeta(3)
\end{equation}
\begin{equation}
\frac{1}{81p^4}\zeta(2)<\varphi_2(r,p)<\frac{1}{p^4}\zeta(2)
\end{equation}
\end{lemma}
\begin{proof}
To get the proof of Lemma \ref{lemfurth1}, just remark that for
$i\geq1$ we have
\[\frac{1}{81i^4p^4}<\frac{1}
{\left(ip+r+1\right)\left(ip+p+r+1\right)\left(ip+r+2\right)\left(ip+p+r\right)}<\frac{1}{i^4p^4}\]
and for $0\leq r\leq p-1$
\[\frac{1}{2p^2(p+1)(2p-1)}<\frac{1}
{\left(r+1\right)\left(p+r+1\right)\left(r+2\right)\left(p+r\right)}<\frac{1}{2p(p+1)}\]
furthermore the last inequalities are deduced.
\end{proof}
For simplifying calculus, let us denoting the function $\phi(r,p)$
to be \[\phi(r,p)=p\varphi_2(r,p)+(p+2r+2)\varphi_1(r,p)+
(2r+2)\varphi_0(r,p).\] Then in one hand we have
\[\phi(r,p)<\frac{1}{p^3}\zeta(2)+\frac{p+2}{p^4}\zeta(3)+\frac{1}{p(p+1)}+\frac{2}{p^4}\zeta(4)+\frac{1}{p}\left(
\frac{2}{p^3}\zeta(3)+\frac{1}{p+1}+\frac{2}{p^3}\zeta(4)\right)r\]
and
\begin{eqnarray*}
\sum_{r=0}^{p-1}\phi(r,p)&<&\frac{1}{p^2}\zeta(2)+\frac{p+2}{p^3}\zeta(3)+\frac{1}{p+1}+\frac{2}{p^3}\zeta(4)\\
&+&\left(\frac{1}{p^3}\zeta(3)+\frac{1}{2(p+1)}+\frac{1}{p^3}\zeta(4)\right)(p-1).
\end{eqnarray*}
Which inducts that
\begin{eqnarray}\label{eq1}\\
\nonumber
c_0\left(\frac{1}{p}\right)<\frac{\left(p-1\right)\left(p-2\right)}{2\pi}\left[\frac{\zeta(2)+2\zeta(3)+\zeta(4)}{p}+\frac{2\left(\zeta(3)+\zeta(4)\right)}{p^2}
+1+\frac{1}{2}p \right]
\end{eqnarray}
In other hand
\begin{eqnarray*}
\phi(r,p)&>&\frac{\zeta(2)}{81p^3}+\frac{(p+2)\zeta(3)}{81p^4}+\frac{1}{p^2(p+1)(2p-1)}+\frac{2\zeta(4)}{81p^4}\\
&+&\left(\frac{2\zeta(3)}{81p^4}+\frac{1}{p^2(p+1)(2p-1)}+\frac{2\zeta(4)}{81p^4}\right)r
\end{eqnarray*}
and
\begin{eqnarray*}
\sum_{r=0}^{p-1}\phi(r,p)&>&\frac{\zeta(3)}{81p^2}+\frac{(p+2)\zeta(3)}{81p^3}+\frac{1}{p(p+1)(2p-1)}+\frac{2\zeta(4)}{81p^3}\\
&+&\left(\frac{\zeta(3)}{81p^3}+\frac{1}{2p(p+1)(2p-1)}+\frac{\zeta(4)}{81p^3}\right)\left(p-1\right)
\end{eqnarray*}
Furthermore
\begin{eqnarray}\label{eq2}
\nonumber
c_0\left(\frac{1}{p}\right)&>&\frac{\left(p-1\right)\left(p-2\right)}{2\pi}\left[\frac{\zeta(3)}{81p}+\frac{(p+2)\zeta(3)}{81p^2}+
\frac{1}{(p+1)(2p-1)}+\frac{2\zeta(4)}{81p^2}\right]\\
&+&\frac{\left(p-1\right)^2\left(p-2\right)}{2\pi}\left(\frac{\zeta(3)}{81p^2}+\frac{1}{2(p+1)(2p-1)}+\frac{\zeta(4)}{81p^2}\right)
\end{eqnarray}
The passage to the limit in two last inequalities \eqref{eq1} and
\eqref{eq2} states that
\begin{equation*}
0\leq\lim_{p\to+\infty}\frac{1}{p^3}c_0\left(\frac{1}{p}\right)\leq\frac{1}{2}.
\end{equation*}
\subsection{New property of the floor function}
Returning back to own recursive formula of $b_k$ and using its
expression \eqref{lembk} in Lemma \ref{genera} we obtain
\begin{theorem}\label{th3}
For $2\leq k\leq p-1$ and $k=p+1$ we have
\begin{equation}\label{equa1th3}
\sum_{i=0}^{2}{2\choose
i}(-1)^{i}\left(k-i+1-\frac{p}{2}\left\lfloor\frac{k-i}{p}\right
\rfloor\right)\left(\left\lfloor\frac{k-i}{p}\right\rfloor+1\right)=0,
\end{equation}
at $k=p$
\begin{equation}\label{equa2th3}
\sum_{i=0}^{2}{2\choose
i}(-1)^{i}\left(p-i+1-\frac{p}{2}\left\lfloor\frac{p-i}{p}\right
\rfloor\right)\left(\left\lfloor\frac{p-i}{p}\right\rfloor+1\right)=1
\end{equation}
and for $k\geq p+2$;
\begin{equation}\label{equa3th3}
\sum_{j=0}^{1}\sum_{i=0}^{2}(-1)^{i+j}{2\choose
i}\left[\left(k-jp-i+1-\frac{p}{2}\left\lfloor\frac{k-jp-i}{p}\right\rfloor\right)
\left(\left\lfloor\frac{k-jp-i}{p}\right\rfloor+1\right)\right]=0
\end{equation}
\end{theorem}
\begin{proof}
We can write for $2\leq k\leq p+1$
\[b_k-2b_{k-1}+b_{k-2}=\sum_{i=0}^{2}{2\choose i}(-1)^ib_{k-i}\]
and for $k\geq p+2$;
\[b_k-2b_{k-1}+b_{k-2}-b_{k-p}+2b_{k-p-1}-b_{k-p-2}=\sum_{j=0}^{1}\sum_{i=0}^{2}{2\choose
i}(-1)^{i+j}b_{k-jp-i}.\] But from the recursive formulae
\eqref{recurse1}, \eqref{recurse2} and \eqref{recurse3} of the
sequence $b_k$ we deduce respectively the formulae \eqref{equa1th3},
\eqref{equa2th3} and \eqref{equa3th3}.
\end{proof}

\subsection{sum of some numerical series}
Inspired from the results $c_0(\frac{1}{3})=\frac{1}{3\sqrt{3}}$,
$c_0(\frac{1}{4})=\frac{1}{2}$ and
$c_0(\frac{1}{6})=\frac{7}{3\sqrt{3}}$ in \cite{GOUBI1} and
$c_0(\frac{1}{5})=\frac{\left(\sqrt{5}-1\right)\sqrt{5
-\sqrt{5}}+3\left(\sqrt{5}+1\right)\sqrt{5 +\sqrt{5}}}{10\sqrt{10}}$
in \cite{Bettin3}, we conclude that
\begin{equation}
\sum_{k\geq0}\frac{\left(k+1-\frac{3}{2}\left\lfloor\frac{k}{3}\right
\rfloor\right)\left(\left\lfloor\frac{k}{3}\right\rfloor+1\right)}{\left(k+1\right)\left(k+4\right)\left(k+2\right)\left(k+3\right)}=\frac{\pi}{18\sqrt{3}},
\end{equation}
\begin{equation}
\sum_{k\geq0}\frac{\left(k+1-2\left\lfloor\frac{k}{4}\right
\rfloor\right)\left(\left\lfloor\frac{k}{4}\right\rfloor+1\right)}{\left(k+1\right)\left(k+5\right)\left(k+2\right)\left(k+4\right)}=\frac{\pi}{48},
\end{equation}
\begin{equation}
\sum_{k\geq0}\frac{\left(k+1-\frac{5}{2}\left\lfloor\frac{k}{5}\right
\rfloor\right)\left(\left\lfloor\frac{k}{5}\right\rfloor+1\right)}{\left(k+1\right)\left(k+6\right)\left(k+2\right)\left(k+5\right)}=
\frac{\left(\sqrt{5}-1\right)\sqrt{5-\sqrt{5}}+3\left(\sqrt{5}+1\right)\sqrt{5+\sqrt{5}}}{600\sqrt{10}}\pi
\end{equation}
and
\begin{equation}
\sum_{k\geq0}\frac{\left(k+1-3\left\lfloor\frac{k}{6}\right
\rfloor\right)\left(\left\lfloor\frac{k}{6}\right\rfloor+1\right)}{\left(k+1\right)\left(k+7\right)\left(k+2\right)\left(k+6\right)}=\frac{7\pi}{360\sqrt{3}}.
\end{equation}

\end{document}